\documentclass[11pt]{article}

\usepackage{dsfont}

\usepackage{psfrag}

\usepackage{calrsfs}
\usepackage{mathrsfs}
\usepackage{pdfsync}
\usepackage{amsmath, amsthm, amssymb, amsfonts}
\usepackage[shortlabels]{enumitem}

\usepackage{url}
\usepackage{float}

\usepackage[usenames]{color}

\usepackage{tikz}

\usetikzlibrary{arrows,decorations.pathmorphing,backgrounds,positioning,fit,automata}

\usetikzlibrary{arrows}
\usetikzlibrary{petri}
\usetikzlibrary{topaths}

\usepackage{indentfirst,calc,euscript}
\usepackage{setspace}
\usepackage[reals]{layout}
\usepackage{xr}
\usepackage{amscd}

\usepackage{sgame}
\usepackage{subfigure}
\usepackage{slashbox}

\usepackage[colorlinks=true,breaklinks=true,bookmarks=true,urlcolor=blue,
     citecolor=blue,linkcolor=blue,bookmarksopen=false,draft=false]{hyperref}

\usepackage{geometry}
\newcommand{\ignore}[1]{}

\newtheorem{theorem}{Theorem}

\newtheorem{lemma}[theorem]{Lemma}

\newtheorem{proposition}[theorem]{Proposition}

\theoremstyle{definition}

\newtheorem{definition}[theorem]{Definition}

\newtheorem{remark}[theorem]{Remark}

\newtheorem{assumption}{Assumption}

\newcounter{step}

\newenvironment{step}[1][Step \thestep]{\refstepcounter{step}  \begin{trivlist}
\item[\hskip \labelsep {\bfseries #1}]}{\end{trivlist}}

\numberwithin{equation}{section}

\numberwithin{theorem}{section}

\newcommand{\m}{\mathbb}

\author{Bruno Ziliotto  \thanks{TSE (GREMAQ, Universit\'{e} Toulouse 1 Capitole), 21 all\'{e}e de Brienne, 31000 Toulouse,
France. \newline E-mail: ziliotto@math.cnrs.fr}} 
\title{A Tauberian theorem for nonexpansive operators and applications to zero-sum stochastic games}
\begin{document}
\maketitle
\bibliographystyle{plain}
\begin{abstract}
We prove a Tauberian theorem for nonexpansive operators, and apply it to the model of zero-sum stochastic game. Under mild assumptions, we prove that the value of the $\lambda$-discounted game $v_{\lambda}$ converges uniformly when $\lambda$ goes to 0 if and only if the value of the $n$-stage game $v_n$ converges uniformly when $n$ goes to infinity. This generalizes the Tauberian theorem of Lehrer and Sorin \cite{LS92} to the two-player zero-sum case. We also provide the first example of a stochastic game with public signals on the state and perfect observation of actions, with finite state space, signal sets and action sets,  in which for some initial state $k_1$ known by both players, $(v_{\lambda}(k_1))$ and $(v_n(k_1))$ converge to distinct limits.
\end{abstract}
\section*{Introduction}
Zero-sum stochastic games were introduced by Shapley \cite{SH53}. In this model, two players repeatedly play a zero-sum game, which depends on the state of nature. At each stage, a new state of nature is drawn from a distribution based on the actions of players and the state of the previous stage. The state of nature is announced to both players, along with the actions of the previous stage. 
There are several ways to evaluate the payoff in a stochastic game. For $n \in \m{N}^*$, the payoff in the $n-stage \ game$ is the Cesaro mean $n^{-1} \sum_{m=1}^n g_m$, where $g_m$ is the payoff at stage $m \geq 1$. For $\lambda \in (0,1]$, the payoff in the $\lambda-discounted \ game$ is the Abel mean $\sum_{m \geq 1} \lambda(1-\lambda)^{m-1} g_m$. Under mild assumptions, the $n$-stage game and the $\lambda$-discounted game have a value, denoted respectively by $v_n$ and $v_{\lambda}$ (see Maitra and Parthasarathy \cite{MP70} and Nowak \cite{N85}).

A huge part of the literature focuses on the existence of the limit of $v_{n}$ when $n$ goes to infinity, and of the limit of $v_{\lambda}$ when $\lambda$ goes to 0. Bewley and Kohlberg \cite{BK76} proved that $(v_n)$ and $(v_{\lambda})$ converge to the same limit, when the state space and action sets are finite. For Markov Decision Processes, this result extends to the case of compact state space, infinite action set and 1-Lipschitz transition (see Rosenberg, Solan and Vieille \cite{RSV02} and Renault \cite{R11}). For absorbing games, this result extends to the case of infinite state space, compact action sets and continuous payoff and transition functions (see Mertens, Neyman and Rosenberg \cite{MNR09}). Vigeral \cite{vigeral13} provided an example of a stochastic game with finite state space and compact action sets in which neither $(v_n)$ nor $(v_{\lambda})$ converges. A natural question is whether the convergence of $(v_{n})$ implies the convergence of $(v_{\lambda})$, and conversely. When $(v_{\lambda})$ is absolutely continuous with respect to $\lambda$, Neyman \cite[Appendix C, p.177]{sorin02b} proved that $(v_n)$ converges to the limit of $(v_{\lambda})$. In the dynamic programming framework, Lehrer and Sorin \cite{LS92} proved that $(v_n)$ converges uniformly (with respect to the initial state) if and only if $(v_{\lambda})$ converges uniformly, and that when uniform convergence holds, the two limits coincide \footnote{For a proof of this result in a continuous-time model, see Oliu Barton and Vigeral \cite{BV13}. Still in continuous time, in a very recent independent paper, Khlopin \cite{K14} has generalized this result to the two-player case.}. This result does not hold when uniform convergence is replaced by pointwise convergence (see Sorin \cite[Chapter 1, p. 9-10]{sorin02b}). In the two-player case, Li and Venel \cite{LV14} proved that for recursive games (which are stochastic games where the payoff is $0$ in nonabsorbing states), $(v_n)$ converges uniformly if and only if $(v_{\lambda})$ converges uniformly, and that when uniform convergence holds the two limits are equal. The generalization of this result to stochastic games was open.

Mertens, Sorin and Zamir \cite[Chapter IV]{MSZ} have introduced a general model of stochastic game with signals, in which players neither observe the state nor the action of their opponent, but instead observe at every stage a signal correlated to the current state and the actions which have just been played (state space, action and signal sets are assumed to be finite). Ziliotto \cite{Z13} provided an example of a stochastic game with public signals on the state and perfect observation of actions, such that $(v_n)$ and $(v_{\lambda})$ fail to converge (for special classes of stochastic games with signals in which $(v_n)$ and $(v_{\lambda})$ converge to the same limit, see \cite{GOV13,N08,R06,R12,rosenberg00,RSV02,RSV04,V10}). The question of the relation between the convergence of $(v_n)$ and $(v_{\lambda})$ was also open. By Mertens, Sorin and Zamir \cite[Chapter III]{MSZ}, one can associate to any stochastic game with signals an auxiliary stochastic game with perfect observation of the state and actions, which has the same $n$-stage and $\lambda$-discounted values. The state space of this auxiliary game is infinite and compact metric, and is the set of infinite higher-order beliefs of players about the state. That is why in this paper we study first stochastic games, then apply our results to stochastic games with signals.

The contribution of this paper is twofold. First, it generalizes both the result of Lehrer and Sorin \cite{LS92} and Li and Venel \cite{LV14} to stochastic games. We consider any stochastic game (with possibly infinite set space and action sets) in which for all $n \in \m{N}^*$ and $\lambda \in (0,1]$, $(v_{n})$ and $(v_{\lambda})$ exist and satisfy Shapley equations, and prove the following Tauberian theorem: $(v_n)$ converges uniformly if and only if $(v_{\lambda})$ converges uniformly, and when uniform convergence holds the two limits are equal. This theorem applies to many standard models in the literature: dynamic programming, stochastic games with finite state space and compact action sets, stochastic games with signals, hidden stochastic games, and Markov chain games. The proof of our result relies on the operator approach, introduced by Rosenberg and Sorin \cite{RS01}. This approach relies on the fact that the values of the $n$-stage game and the $\lambda$-discounted game satisfy a functional equation, called the Shapley equation (see Shapley \cite{SH53}). The properties of the associated nonexpansive operator can be exploited to infer convergence properties of $(v_n)$ and $(v_{\lambda})$ (see Rosenberg and Sorin \cite{RS01}). Thus, we start by proving a Tauberian theorem for nonexpansive operators, and then apply it to stochastic games.

Second, this paper provides the first example of a stochastic game with public signals on the state and perfect observation of the actions (hidden stochastic game), with finite state space, signal sets and action sets,  in which for some initial state $k_1$ known by both players, $(v_{\lambda}(k_1))$ and $(v_n(k_1))$ converge to distinct limits (note that in the example in Sorin \cite[Chapter 1, p. 9-10]{sorin02b}, the state space is infinite and not compact). An example of a stochastic game with finite state space, compact action sets, perfect observation of the state and actions and having the same property can be deduced from this example. Thus, our example shows that as soon as the state is imperfectly observed, or the state space is not finite, or the action sets are not finite, there is no link between the convergence of $(v_{\lambda}(k_1))$ and $(v_n(k_1))$, where $k_1$ is some initial state.

The paper is organized as follows. In the first section, a Tauberian theorem for nonexpansive operators is stated and proved. In the second section, a Tauberian theorem for stochastic games is deduced from the first section. In the third section, particular cases of stochastic games are considered. The fourth section presents the aforementioned example.

\section{Nonexpansive operators}
Let $(X, \left\| . \right\|$) be a Banach space, and $\Psi:X \rightarrow X$ be a nonexpansive mapping, that is:
\begin{equation*}
\forall \ (f,g) \in X^2 \quad \left\|\Psi(f)-\Psi(g) \right\| \leq \left\|f-g\right\|.
\end{equation*}
By a standard fixed point argument (see Sorin \cite[Appendix C]{sorin02b}), there exists a bounded family $(v_{\lambda})_{\lambda \in (0,1]}$ such that for all $\lambda \in (0,1]$
\begin{equation} \label{recop}
v_{\lambda}=\lambda \Psi((1-\lambda) \lambda^{-1} v_{\lambda}).
\end{equation}

For $n \in \m{N}^*$, define
\begin{equation} \label{recopn}
v_n:=n^{-1} \Psi^n(0),
\end{equation}
where $\Psi^n$ is the $n$-th iterate of $\Psi$.
Because $\Psi$ is nonexpansive, $(v_n)_{n \geq 1}$ is bounded.

Kohlberg and Neyman \cite{KN81} provided conditions under which $\lim_{n \rightarrow +\infty} v_n $ and $\lim_{\lambda \rightarrow 0} v_{\lambda}$ exist. In this section, we investigate the link between the existence of $\lim_{n \rightarrow +\infty} v_n$ and
$\lim_{\lambda \rightarrow 0} v_{\lambda}$.
We make the following assumption:
\begin{assumption} \label{assop}
There exists $C>0$ such that for all $\lambda, \lambda' \in (0,1]$, $f \in X$,
\begin{equation*}
\left\| \lambda \Psi(\lambda^{-1} f)-\lambda' \Psi(\lambda'^{-1} f) \right\| \leq C \left|\lambda-\lambda' \right|.
\end{equation*}
\end{assumption}
\begin{remark} \label{example}
An important class of operators which satisfy Assumption 1 is the following. Let $K$ be any set, and $X$ be the set of bounded real-valued functions defined on $K$, equipped with the uniform norm. Consider two sets $S$ and $T$, and a family of linear forms $(P_{k,s,t})_{(k,s,t) \in K \times S\times T}$ on $X$, such that for all $(k,s,t)$, $P_{k,s,t}$ is of norm smaller than 1. Let $g:K \times S \times T \rightarrow \m{R}$ be a bounded function. Define $\Psi :X \rightarrow X$ by $\Psi(f)(k):=\sup_{s \in S} \inf_{t \in T} \left\{g(k,s,t)+P_{k,s,t}(f) \right\}$, for all $f \in X$ and $k \in K$. This class includes \textit{Shapley operators} (see Neyman \cite[p.397-415]{NS03b}): this corresponds to the case where $K$ is the state space of some zero-sum stochastic game, $S$ (resp. $T$) is the set of mixed actions of Player 1 (resp. 2), $k$ is the current state, and $P_{k,s,t}(f)$ is the expectation of $f(k')$ under mixed actions $s$ and $t$, where $k'$ is the state at next stage. 
Under suitable assumptions, for all $n \in \m{N}^*$ and $\lambda \in (0,1]$, $v_n$ and $v_{\lambda}$ are respectively the value of the $n$-stage game and the value of the $\lambda$-discounted game. This point will be useful in Sections \ref{app} and \ref{ex}.
\end{remark}
We now state a Tauberian theorem for nonexpansive operators satisfying Assumption 1.
\begin{theorem} \label{taubgen}
Under Assumption 1, the two following statements are equivalent:
\begin{enumerate}[(a)]
\item The sequence $(v_n)_{n \geq 1}$ converges when $n$ goes to infinity.
\item The mapping $\lambda \rightarrow v_{\lambda}$ has a limit when $\lambda$ goes to 0.
\end{enumerate}
Moreover, when these statements hold, we have $\lim_{n \rightarrow+\infty} v_n=\lim_{\lambda \rightarrow 0} v_{\lambda}$.
\end{theorem}


The remainder of this section is dedicated to the proof of the theorem. 
\begin{definition}
Let $\lambda \in (0,1]$ and $n \in \m{N}^*$. The operator $\Psi^{n}_{\lambda}: X \rightarrow X$ is defined recursively by 
$\Psi^{0}_{\lambda}(f):=f$ for all $f \in X$, and for $n \geq 1$:
\begin{equation*}
\forall \ f \in X \quad \Psi^{n}_{\lambda}(f):=\lambda \Psi((1-\lambda) \lambda^{-1} \Psi^{n-1}_{\lambda}(f)).
\end{equation*}
Note that equation (\ref{recop}) writes
\begin{equation*}
v_{\lambda}=\Psi^1_{\lambda}(v_{\lambda}).
\end{equation*}
\end{definition}
\begin{lemma} \label{iterate} 
Let $f,g \in X$,  $\lambda \in (0,1]$, $n \in \m{N}^*$ and $t \in \left\{1,2,...,n\right\}$. Then
\begin{enumerate}[(i)]
\item \label{it1}
\begin{equation*}
\left\|\Psi^{t}_{\lambda}(f)-\Psi^{t}_{\lambda}(g)\right\| \leq (1-\lambda)^{t} \left\|f-g \right\|
\end{equation*}
\item \label{it2}
\begin{equation*}
\left\|\Psi^{t}_{n^{-1}}(f)-n^{-1}\Psi^{t}((n-t) f)\right\| \leq (C+\left\|f \right\|) \left[t n^{-1}-1+(1-n^{-1})^{t} \right].
\end{equation*}
\end{enumerate}
\end{lemma}
\proof{Proof}
\begin{enumerate}[(i)] 
\item This follows from the nonexpansiveness of $\Psi$.
\\
\item
We have
\begin{eqnarray*}
\left\|\Psi^{t}_{n^{-1}}(f)-n^{-1}\Psi^{t}((n-t)f)\right\|&\leq& 
\left\|(1-n^{-1}) \Psi_{n^{-1}}^{t-1}(f)-n^{-1} \Psi^{t-1}((n-t)f) \right\|
\\
&\leq&
C \left[n^{-1}-(1-n^{-1})n^{-1}\right]+\left\|(1-n^{-1})^2 \Psi_{n^{-1}}^{t-2}(f)-n^{-1} \Psi^{t-2}((n-t)f) \right\|
\\
&\leq& C \sum_{m=1}^t (n^{-1}-(1-n^{-1})^{m-1} n^{-1})+
\left\|(1-n^{-1})^t f-n^{-1}(n-t)f \right\|
\\
&=& C \left[t n^{-1}-1+(1-n^{-1})^t \right]+\left[(1-n^{-1})^t-n^{-1}(n-t)\right] \left\|f \right\|
\\
&=& (C+\left\|f \right\|) \left[t n^{-1}-1+(1-n^{-1})^{t} \right].
\end{eqnarray*}
The first inequality stems from the nonexpansiveness of $\Psi$. In the second inequality, we applied Assumption 1 for $\lambda=(1-n^{-1})n^{-1}$, $\lambda'=n^{-1}$ and $\widetilde{f}=(1-n^{-1})^2 \Psi^{t-2}_{n^{-1}}(f)$, and used the nonexpansiveness of $\Psi$. Applying successively Assumption 1 for $\lambda=(1-n^{-1})^{m-1} n^{-1}$, $\lambda'=n^{-1}$ and $\widetilde{f}=(1-n^{-1})^{m} \Psi^{t-m}_{n^{-1}}(f)$ ($m \in \left\{1,...,t \right\}$) together with the nonexpansiveness of $\Psi$ yields the third inequality.
\end{enumerate}
\endproof
We now prove that $(a)$ implies $(b)$.
\\

$\boldsymbol{(a) \Rightarrow (b)}$
\\

Assume $(a)$. Let $(\lambda,\lambda') \in (0,1]^2$. We have
\begin{eqnarray*}
\left\|v_{\lambda}-v_{\lambda'}\right\|&=&\left\|\lambda \Psi((1-\lambda) \lambda^{-1} v_{\lambda})-\lambda' \Psi((1-\lambda') {\lambda'}^{-1} v_{\lambda'})\right\|
\\
&\leq& C\left|\lambda-\lambda'\right|+\left\|\lambda' \Psi((1-\lambda) {\lambda'}^{-1} v_{\lambda})-\lambda' \Psi((1-\lambda') {\lambda'}^{-1} v_{\lambda'})\right\|
\\
&\leq& (C+\left\|v_{\lambda}\right\|) \left|\lambda-\lambda'\right|+(1-\lambda') \left\|v_{\lambda}-v_{\lambda'}\right\|.
\end{eqnarray*}
In the first inequality, we applied Assumption 1 to $f=(1-\lambda)v_{\lambda}$, and in the second inequality, we applied twice the nonexpansiveness of $\Psi$.
We deduce the existence of $A>0$ such that for all $(\lambda,\lambda') \in (0,1]^2$, $\left\|v_{\lambda}-v_{\lambda'} \right\| \leq A \left|\lambda-\lambda'\right| {\lambda'}^{-1}$. Consequently, in order to prove $(b)$, it is sufficient to prove that $(v_{n^{-1}})_{n \geq 1}$ converges when $n$ goes to infinity.

By $(a)$, there exists $v^* \in X$ such that $(v_n)_{n \geq 1}$ converges to $v^*$.
Let $\epsilon \in (0,1/4)$. Let $N_0 \in \m{N}^*$ such that for all $n \geq N_0$,
\begin{equation} \label{limvn}
\left\|v_n-v^* \right\| \leq \epsilon^2/2.
\end{equation}
 Let $n \geq \epsilon^{-2}N_0$, $\lambda:=n^{-1}$, and $t:=\lfloor \epsilon n \rfloor$, where $\lfloor x \rfloor$ denotes the integer part of $x$. Equations (\ref{recop}) and (\ref{recopn}) yield

\begin{equation} \label{fix1}
v_{\lambda}=\Psi^{t}_{\lambda}(v_{\lambda}),
\end{equation}
and
\begin{equation} \label{fix2}
v_{n}=n^{-1} \Psi^{t}((n-t)v_{n-t}).
\end{equation}
We have 
\begin{eqnarray*}
\left\|v_{\lambda}-v_{n} \right\| &\leq& \left\|v_{\lambda}-\Psi^{t}_{\lambda}(v_{n-t}) \right\|+
\left\|\Psi^{t}_{\lambda}(v_{n-t})-v_{n} \right\|.
\end{eqnarray*}
Applying first (\ref{fix1}) and Lemma \ref{iterate} \ref{it1}, then (\ref{limvn}), we obtain
\begin{eqnarray*}
\left\|v_{\lambda}-\Psi^{t}_{\lambda}(v_{n-t}) \right\| &\leq& (1-\lambda)^t \left\|v_{\lambda}-v_{n-t} \right\|  
\\
&\leq& (1-\lambda)^t \left\|v_{\lambda}-v_{n}\right\| + (1-\lambda)^t \left\|v_{n}-v_{n-t} \right\|
\\
&\leq& (1-\lambda)^t \left\|v_{\lambda}-v_{n}\right\| + \epsilon^2.
\end{eqnarray*}
Let $\displaystyle M:=C+\sup_{n \in \m{N}} \left\|v_n\right\|$.
Equality (\ref{fix2}) and Lemma \ref{iterate} \ref{it2} yield
\begin{eqnarray*}
\left\|\Psi^{t}_{\lambda}(v_{n-t})-v_{n} \right\| &\leq& (C+\left\|v_{n-t} \right\|)\left[t n^{-1}-1+(1-n^{-1})^t\right]
\\
&\leq& M \left(\epsilon-1+e^{-\epsilon+\epsilon^2}\right).
\end{eqnarray*}
The last two inequalities yield
\begin{eqnarray*}
\left\|v_{\lambda}-v_{n} \right\| &\leq& (1-\lambda)^{\epsilon n-1} \left\|v_{\lambda}-v_{n} \right\|+\epsilon^2+M \left(\epsilon-1+e^{-\epsilon+\epsilon^2}\right)
\\
&\leq& e^{-\epsilon+\epsilon^{2}} \left\|v_{\lambda}-v_{n} \right\| +\epsilon^2+M \left(\epsilon-1+e^{-\epsilon+\epsilon^2}\right).
\end{eqnarray*}
We deduce that
\begin{equation*}
\left\|v_{\lambda}-v_{n} \right\| \leq \left[\epsilon^2+M \left(\epsilon-1+e^{-\epsilon+\epsilon^2}\right) \right]\left(1-e^{-\epsilon+\epsilon^{2}}\right)^{-1}.
\end{equation*}
The right-hand side goes to 0 when $\epsilon$ goes to 0, thus $(b)$ holds.
 \\

$\boldsymbol{(b) \Rightarrow (a)}$
\\

Assume $(b)$.
There exists $\beta \in (0,1)$ such that for all $\lambda \in (0,\beta]$, we have
\begin{equation} \label{limvl}
\left\|v_{\lambda}-v^* \right\| \leq \epsilon^2/2.
\end{equation}
Let $\epsilon_0 \in (0,1)$ such that for all $\epsilon \leq \epsilon_0$, $e^{-\epsilon} \leq 1-\epsilon+\epsilon^2$. Fix $\epsilon \in (0,\epsilon_0/2]$, and define
$r_0:=\lfloor \epsilon^{-3/2} \rfloor$. Let $N \geq 1$ such that $\lfloor (1-\epsilon)^{r_0-1}N \rfloor \geq (\beta \epsilon)^{-1}$. 
Let $n \geq N$. For $r \in \m{N}^*$, define $n_r:=\lfloor(1-\epsilon)^{r-1} n \rfloor$ and $\lambda_r:=1/n_r$. The following assertions hold:
\begin{lemma} \label{l2} \mbox{}
\begin{enumerate}[(i)]
\item
$\forall \ r \in \left\{1,...,r_0 \right\} \ \lambda_r \leq \beta \epsilon$
\item
$\forall \ r \in \left\{1,...,r_0-1 \right\} \ \left(1-1/{n_r} \right)^{n_r-n_{r+1}}-n_{r+1}/{n_r} \leq 4 \epsilon^2$
\item \label{vl}
$\forall \ r \in \left\{1,...,r_0-1 \right\} \ \left\|v_{\lambda_r}-v_{\lambda_{r+1}} \right\| \leq \epsilon^2$
\end{enumerate}
\end{lemma}
\proof{Proof} \mbox{}
\begin{enumerate}[(i)]
\item \label{inw2}
Let $r \in \left\{1,...,r_0 \right\}$. We have $\lfloor (1-\epsilon)^{r-1} n \rfloor \geq (\beta \epsilon)^{-1}$, thus $\lambda_r \leq \beta \epsilon$.
\\
\item
Let $r \in \left\{1,...,r_0-1 \right\}$. We have
\begin{eqnarray*}
(1-1/{n_r})^{n_r-n_{r+1}} &\leq& e^{-(n_r-n_{r+1})/{n_r}}
\\
&\leq& 1-(n_r-n_{r+1})/{n_r}+\left[(n_r-n_{r+1})/{n_r}\right]^2
\\
&\leq& n_{r+1}/n_r+(1-n_{r+1}/n_r)^2
\\
&\leq& n_{r+1}/n_r+(\epsilon+1/n_r)^2
\\
&\leq& n_{r+1}/n_r+4\epsilon^2.
\end{eqnarray*}
\item
It is a direct consequence of (\ref{limvl}) and \ref{inw2}.
\end{enumerate}
\endproof
\vspace{0.25cm}
Let $r \in \left\{1,...,r_0-1 \right\}$.
 Equations (\ref{recop}) and (\ref{recopn}) yield
\begin{equation} \label{fix3}
v_{n_r}=(n_r)^{-1}\Psi^{n_r-n_{r+1}}(n_{r+1}v_{n_{r+1}})
\end{equation}
and
\begin{equation} \label{fix4}
v_{\lambda_r}=\Psi^{n_r-n_{r+1}}_{\lambda_r}(v_{\lambda_r}).
\end{equation}
We have
\begin{equation*}
\left\|v_{n_r}-v_{\lambda_r} \right\| \leq \left\|v_{n_r}-(n_r)^{-1} \Psi^{n_r-n_{r+1}}(n_ {r+1}v_{\lambda_r}) \right\|+
\left\|(n_r)^{-1} \Psi^{n_r-n_{r+1}}(n_{r+1}v_{\lambda_r})-v_{\lambda_r} \right\|.
\end{equation*}
Applying first (\ref{fix3}) and then Lemma \ref{l2} \ref{vl} yields
\begin{eqnarray*}
\left\|v_{n_r}-(n_r)^{-1} \Psi^{n_r-n_{r+1}}(n_ {r+1}v_{\lambda_r}) \right\|&\leq& (n_r)^{-1} n_{r+1} \left\|v_{n_{r+1}}-v_{\lambda_{r}} \right\|
\\
&\leq& (n_r)^{-1} n_{r+1} 
\left(\left\|v_{n_{r+1}}-v_{\lambda_{r+1}} \right\|
+\epsilon^2 \right).
\end{eqnarray*}
Let $M:=C+\sup_{n \in \m{N}^*} \left\|v_n \right\|+\sup_{\lambda \in (0,1]} \left\|v_{\lambda} \right\|+1$. Equality (\ref{fix4}) and Lemma \ref{iterate} \ref{it2} yield
\begin{eqnarray*}
\left\|(n_r)^{-1} \Psi^{n_r-n_{r+1}}(n_{r+1}v_{\lambda_r})-v_{\lambda_r} \right\| &\leq& 
(C+\left\|v_{\lambda_r}\right\|) \left[(n_r-n_{r+1})/n_r-1+(1-1/n_r)^{n_r-n_{r+1}} \right]
\\
&\leq& M \left[(1-1/{n_r})^{n_r-n_{r+1}}-n_{r+1}/{n_r} \right]
\\
&\leq& 4 \epsilon^2 M.
\end{eqnarray*}
We deduce that
\begin{eqnarray*}
n_r \left\|v_{n_r}-v_{\lambda_r} \right\| &\leq&   n_{r+1}\left(\left\|v_{n_{r+1}}-v_{\lambda_{r+1}} \right\|+\epsilon^2\right)+4 \epsilon^2 M n_r
\\
&\leq& n_{r+1}\left\|v_{n_{r+1}}-v_{\lambda_{r+1}} \right\| +5 \epsilon^2 M n_r.
\end{eqnarray*}
Summing from $r=1$ to $r=r_0-1$ yields
\begin{eqnarray*}
\left\|v_{n}-v_{n^{-1}} \right\| &\leq& n_{r_0} n^{-1} \left\|v_{n_{r_0}}-v_{\lambda_{r_0}} \right\|+5 \epsilon^2 M r_0
\\
&\leq& 2(1-\epsilon)^{\lfloor\epsilon^{-3/2} \rfloor-1}M+5 \epsilon^{1/2}M.
\end{eqnarray*}
The right-hand side goes to 0 as $\epsilon$ goes to 0, thus $(a)$ holds.

Note that the proof of the two implications show that when $(a)$ and $(b)$ hold, we have $\lim_{n \rightarrow +\infty} v_n=\lim_{\lambda \rightarrow 0} v_{\lambda}$.
\section{Applications to zero-sum stochastic games} \label{app}
In this section, we apply Theorem \ref{taubgen} to zero-sum stochastic games.
\subsection{Dynamic programming}
A dynamic programming problem is described by a state space $K$, a nonvoid correspondence $F :K \rightrightarrows K$, and a bounded payoff function $g:K \rightarrow \m{R}$. 

The problem proceeds as follows. Given an initial state $k_1 \in K$, at each stage, the decision-maker chooses $k_{m+1} \in F(k_m)$, and gets the stage payoff $g(k_m)$. For $\lambda \in (0,1]$ (resp. $n \in \m{N}^*$) in the $\lambda$-discounted problem (resp. $n$-stage problem), the decision-maker maximizes the total payoff $\sum_{m \geq 1} \lambda(1-\lambda)^{m-1} g(k_m)$ (resp. $n^{-1} \sum_{m=1}^n g(k_m)$). 

A \textit{strategy} for the decision-maker assigns a decision $k_{m+1} \in F(k_m)$ to each finite history $(k_1,k_2,...,k_m)$. We denote respectively $v_{\lambda}(k_1)$ and $v_n(k_1)$ the value of the $\lambda$-discounted problem and $n$-stage problem: $v_{\lambda}(k_1):=\sup_{s \in S} \sum_{m \geq 1} \lambda(1-\lambda)^{m-1} g(k_m)$ and $v_n(k_1):=\sup_{s \in S} n^{-1} \sum_{m=1}^n g(k_m)$, where $S$ is the set of strategies. 

Let $X$ be the set of bounded real-valued functions defined on $K$, equipped with the uniform norm. For $(f,k) \in X \times K$, let 
\begin{equation*}
\Psi(f)(k):=g(k)+ \sup_{k' \in F(k)} f(k').
\end{equation*}
$X$ is a Banach space and $\Psi$ is a nonexpansive operator which satisfies Assumption 1.
Standard dynamic programming gives (see Lehrer and Sorin \cite{LS92}):
\begin{equation*}
v_{\lambda}(k)=\lambda g(k) + (1-\lambda) \sup_{k' \in F(k)} v_{\lambda}(k')=\left[\lambda \Psi((1-\lambda) \lambda^{-1} v_{\lambda})\right](k)
\end{equation*}
and
\begin{equation*}
v_{n}(k)=n^{-1} g(k) + (1-n^{-1}) \sup_{k' \in F(k)} v_{n-1}(k')=\left[n^{-1} \Psi((n-1)v_{n-1}) \right](k).
\end{equation*}
Applying Theorem \ref{taubgen}, we recover the Tauberian theorem proved in Lehrer and Sorin \cite{LS92}: $(v_n)$ converges uniformly on $K$ if and only if $(v_{\lambda})$ converges uniformly on $K$, and when uniform convergence holds, the two limits coincide.
\subsection{Zero-sum stochastic games}
If $\left(C,\mathcal{C}\right)$ is a Borel subset of a Polish space, we denote by $\Delta(C)$ the set of probability measures on $C$, equipped with the weak$^*$ topology.

We use the same framework as in Maitra and Parthasarathy \cite{MP70}. We consider a general model of zero-sum stochastic game, described by a state space $K$ which is a Borel subset of a Polish space, two action sets $I$ and $J$, which are Borel subsets of a Polish space, a Borel measurable transition function $q: K \times I \times J \rightarrow \Delta(K)$, and a bounded Borel measurable payoff function $g:K \times I \times J \rightarrow \m{R}$.

The initial state is $k_1 \in K$, and the stochastic game $\Gamma(k_1)$ which starts in $k_1$ proceeds as follows. At each stage $m \geq 1$, both players choose simultaneously and independently an action, $i_m \in I$ (resp. $j_m \in J$) for Player 1 (resp. 2). The payoff at stage $m$ is $g_m:=g(k_m,i_m,j_m)$.
The state $k_{m+1}$ of stage $m+1$ is drawn from the probability distribution $q(k_m,i_m,j_m)$. Then $(k_{m+1},i_m,j_m)$ is publicly announced to both players.
\\
The set of all possible histories before stage $m$ is
$H_m:=(K \times I \times J)^{m-1} \times K$. A \textit{behavioral strategy} for Player 1 (resp. 2) is a Borel measurable mapping $\displaystyle \sigma:\cup_{m \geq 1} H_m \rightarrow \Delta(I)$ (resp. $\displaystyle \tau:\cup_{m \geq 1} H_m \rightarrow \Delta(J)$).
\\
A triple $(k_1,\sigma,\tau) \in K \times \Sigma \times \mathcal{T}$ induces a probability measure on $H_\infty:=(K \times I  \times J)^{\m{N}^*}$, denoted by $\mathbb{P}^{k_1}_{\sigma,\tau}$. Let $\lambda \in (0,1]$. The $\lambda$-discounted game $\Gamma_{\lambda}(k_1)$ is the game defined by its normal form $(\Sigma,\mathcal{T},\gamma_{\lambda}^{k_1})$, where 
\begin{equation*}
\gamma^{k_1}{\lambda}(\sigma,\tau):=\mathbb{E}^{k_1}_{\sigma,\tau}\left(\sum_{m \geq 1} \lambda (1-\lambda)^{m-1} g_m \right).
\end{equation*}
Let $n \in \m{N}^*$. The $n$-stage game $\Gamma_{n}(k_1)$ is the game defined by its normal form $(\Sigma,\mathcal{T},\gamma_{n}^{k_1})$, where 
\begin{equation*}
\gamma^{k_1}_{n}(\sigma,\tau):=\mathbb{E}^{k_1}_{\sigma,\tau}\left(\frac{1}{n}\sum_{m=1}^n g_m \right).
\end{equation*}
Let $f: K \rightarrow \m{R}$ be a bounded Borel measurable function, and $(k_1 ,x,y) \in K \times \Delta(I) \times \Delta(J)$. Define 
\begin{equation*} \label{deftrans}
\m{E}^{k}_{x,y}(f):=\int_{(k',i,j) \in K \times I \times J} f(k') dq(k,i,j)(k')dx(i) dy(j)
\end{equation*}
and 
\begin{equation*} \label{defpayoff}
\displaystyle g(k,x,y):=\int_{(i,j) \in I \times J} g(k,i,j) dx(i) dy(j) .
\end{equation*}
We make the following assumption:
\begin{assumption}
For all $k_1 \in K$, $\lambda \in (0,1]$ and $n \in \m{N}^*$, the games $\Gamma_{\lambda}(k_1)$ and $\Gamma_{n}(k_1)$ have a value, that is, there exists real numbers $v_{\lambda}(k_1)$ and $v_n(k_1)$ such that:
\begin{equation*}
v_{\lambda}(k_1)=\sup_{\sigma \in \Sigma} \inf_{\tau \in \mathcal{T}} \gamma^{k_1}_{\lambda}(\sigma,\tau)
=\inf_{\tau \in \mathcal{T}} \sup_{\sigma \in \Sigma} \gamma^{k_1}_{\lambda}(\sigma,\tau),
\end{equation*}
and
\begin{equation*}
v_{n}(k_1)=\sup_{\sigma \in \Sigma} \inf_{\tau \in \mathcal{T}} \gamma^{k_1}_{n}(\sigma,\tau)
=\inf_{\tau \in \mathcal{T}} \sup_{\sigma \in \Sigma} \gamma^{k_1}_{n}(\sigma,\tau).
\end{equation*}
Moreover, $v_{\lambda}$ and $v_{n}$ are Borel measurable, and satisfy the following Shapley equations:
\begin{eqnarray*} 
v_{\lambda}(k_1)&=& \sup_{x \in \Delta(I)} \inf_{y \in \Delta(J)}
\left\{ \lambda g(k_1,x,y)+(1-\lambda)\m{E}^{k_1}_{x,y}(v_{\lambda}) \right\}
\\
&=& \inf_{y \in \Delta(J)} \sup_{x \in \Delta(I)} 
\left\{ \lambda g(k_1,x,y)+(1-\lambda)\m{E}^{k_1}_{x,y}(v_{\lambda}) \right\}
\end{eqnarray*}
and
\begin{eqnarray*} \label{dyn11}
v_{n}(k_1)&=& \sup_{x \in \Delta(I)} \inf_{y \in \Delta(J)}
\left\{ n^{-1} g(k_1,x,y)+(1-n^{-1})\m{E}^{k_1}_{x,y}(v_{n-1}) \right\}
\\
&=& \inf_{y \in \Delta(J)} \sup_{x \in \Delta(I)} 
\left\{ n^{-1} g(k_1,x,y)+(1-n^{-1})\m{E}^{k_1}_{x,y}(v_{n-1}) \right\}.
\end{eqnarray*}
\end{assumption}
Let $X$ be the set of bounded Borel measurable functions from $K$ to $\m{R}$, equipped with the uniform norm, and for all $(f,k) \in X \times K$, we define $\Psi(f)(k):=\sup_{x \in \Delta(I)} \inf_{y \in \Delta(J)} \left\{g(k,x,y)+\m{E}^k_{x,y}(f) \right\}$. We make the following assumption:
\begin{assumption}
For all $f \in X$, $\Psi(f)$ is Borel measurable.
\end{assumption}
\begin{remark} \label{compactcase}
When $K$, $I$ and $J$ are compact metric spaces and $q$ and $g$ are jointly continuous, Assumptions 2 and 3 hold. Maitra and Parthasarathy \cite{MP70} and Nowak \cite{N85} provided weaker conditions under which Assumptions 2 and 3 hold.
\end{remark}
Theorem \ref{taubgen} yields the following Tauberian theorem for stochastic games:
\begin{theorem} \label{th2}
Under Assumptions 2 and 3, the two following statements are equivalent:
\begin{enumerate}[(a)]
\item The family of functions $(v_n)_{n \geq 1}$ converges uniformly on $K$ when $n$ goes to infinity.
\item The family of functions $(v_{\lambda})_{\lambda \in (0,1]}$ converges uniformly on $K$ when $\lambda$ goes to 0.
\end{enumerate}
Moreover, when these statements hold, we have $\lim_{n \rightarrow+\infty} v_n=\lim_{\lambda \rightarrow 0} v_{\lambda}$.
\end{theorem}
\proof{Proof}
$X$ is a Banach space, and Assumption 3 ensures that $\Psi$ is well defined from $X$ to $X$. Moreover, $\Psi$ is a nonexpansive operator which satisfies Assumption 1. Thus Theorem \ref{taubgen} applies to $\Psi$. By Assumption 2, the families of values $(v_{\lambda})$ and $(v_n)$ satisfy equations (\ref{recop}) and (\ref{recopn}), and the result is proved.
\endproof
\subsection{Stochastic games with signals}
Assume $K$, $I$ and $J$ to be finite. The previous model can be generalized in the following way. Let $A$ (resp. $B$) be a finite set of signals for Player 1 (resp. 2). Instead of observing the past actions $(i_m,j_m)$ and the future state $k_{m+1}$ at the end of each stage $m$, Player 1 (resp. 2) gets a private signal $a_m$ (resp. $b_m$), which is correlated to $(k_m,i_m,j_m)$ (see Mertens, Sorin and Zamir \cite[Chapter IV]{MSZ} for more details). This defines a \textit{stochastic game with signals}, denoted by $\Gamma$. A behavioral strategy for a player assigns a mixed action to each of his private history (that is, all the actions he has played and all the signals he has received before the current stage). Because $K$, $I$, $J$, $A$ and $B$ are finite, the $\lambda$-discounted game and the $n$-stage game have a value for all $\lambda \in (0,1]$ and $n \in \m{N}^*$.

By Mertens, Sorin and Zamir \cite[Chapter III]{MSZ}, there exists a stochastic game $\widetilde{\Gamma}$ (in the sense of the previous subsection) with perfect observation of the state and actions, which is \textit{equivalent} to $\Gamma$: it has the same $\lambda$-discounted and $n$-stage values. The state space of this auxiliary stochastic game is the set of infinite higher-order beliefs of the players about the state: this is the \textit{universal belief space}, denoted by $\m{B}$. The set $\m{B}$ is compact metric, and Assumptions 2 and 3 are satisfied. Thus Theorem \ref{th2} applies to the auxiliary stochastic game $\widetilde{\Gamma}$.
\begin{proposition} \label{th3}
The two following statements are equivalent:
\begin{enumerate}[(a)]
\item The family of functions $(v_n)_{n \geq 1}$ converges uniformly on $\m{B}$ when $n$ goes to infinity.
\item The family of functions $(v_{\lambda})_{\lambda \in (0,1]}$ converges uniformly on $\m{B}$ when $\lambda$ goes to 0.
\end{enumerate}
Moreover, when these statements hold, we have $\lim_{n \rightarrow+\infty} v_n=\lim_{\lambda \rightarrow 0} v_{\lambda}$.
\end{proposition}
\begin{remark}
It is not known in general if the families $(v_n)_{n \geq 1}$ and $(v_{\lambda})_{\lambda \in (0,1]}$ are equicontinuous. Thus uniform convergence may be difficult to prove, even when pointwise convergence holds. In the examples of the next section, $(v_n)_{n \geq 1}$ and $(v_{\lambda})_{\lambda \in (0,1]}$ are equi-Lipschitz, thus pointwise convergence and uniform convergence are equivalent in these examples.
\end{remark}

\section{Examples of zero-sum stochastic games} \label{ex}
We apply the results of the previous section to several standard examples of zero-sum stochastic games.
\subsection{Stochastic games with compact action sets and finite state space}
We consider the case where the state space is finite, the action sets are compact, and the transition and payoff functions are separately continuous. By the standard minmax theorem, Assumptions 2 and 3 hold (it is a particular case of Maitra and Parthasarathy \cite{MP70}). Because $K$ is finite, uniform convergence and pointwise convergence with respect to the state variable are equivalent. Theorem \ref{th2} yields the following proposition:
\begin{proposition} \label{compactSG}
In a stochastic game with finite state space and compact action sets, the two following statements are equivalent:
\begin{enumerate}[(a)]
\item For all $k_1 \in K$, $(v_n(k_1))$ converges when $n$ goes to infinity.
\item For all $k_1 \in K$, $(v_{\lambda}(k_1))$ converges when $\lambda$ goes to 0.
\end{enumerate}
Moreover, when these statements hold, we have $\lim_{n \rightarrow+\infty} v_n(k_1)=\lim_{\lambda \rightarrow 0} v_{\lambda}(k_1)$ for all $k_1 \in K$.
\end{proposition}
\subsection{Hidden stochastic games} \label{HSG}
Consider the following example of stochastic game with signals. Assume that $K$, $I$ and $J$ are finite, and that players do not observe the current state at each stage (they observe past actions). Instead, they receive a public signal about it, lying in some finite set $A$ (see Renault and Ziliotto \cite{RZ14} for more details). In this particular case, the universal belief space is $\m{B}=\Delta(K)$: this corresponds to the common belief of the players about the state (see Ziliotto \cite{Z13}). Thus, $(v_{\lambda})$ and $(v_n)$ can be considered as families of maps from $\Delta(K)$ to $\m{R}$. They are both equi-Lipschitz, thus for theses families, pointwise convergence and uniform convergence are equivalent. By Proposition \ref{th3}, the following result holds.
\begin{proposition} \label{HSG}
In a hidden stochastic game, the two following statements are equivalent:
\begin{enumerate}[(a)]
\item For all $p_1 \in \Delta(K)$, $(v_n(p_1))$ converges when $n$ goes to infinity.
\item For all $p_1 \in \Delta(K)$, $(v_{\lambda}(p_1))$ converges when $\lambda$ goes to 0.
\end{enumerate}
Moreover, when these statements hold, we have $\lim_{n \rightarrow+\infty} v_n(p_1)=\lim_{\lambda \rightarrow 0} v_{\lambda}(p_1)$ for all $p_1 \in \Delta(K)$.
\end{proposition}

\subsection{Markov chain games with incomplete information on both sides} \label{MC}
Consider the following example of stochastic game with signals. Assume that $K$, $I$ and $J$ are finite, and that the state space is a product $K=C \times D$, such that the two components of the state follow independent Markov chains. Players know the transition and the initial distribution of each Markov chain, but only Player 1 (resp. 2) observes the realization at stage 1 of the first (resp. second) component. From stage 2, they do not observe the state. They observe past actions (see Gensbittel and Renault \cite{GR12} for more details). In this particular case, the equivalent stochastic game with perfect observation of the state has state space $\Delta(C) \times \Delta(D)$, that is, the product of the set of possible beliefs of Player 2 about the initial state of the first Markov chain, and of the set of possible beliefs of Player 1 about the initial state of the second Markov chain. Thus, $(v_{\lambda})$ and $(v_n)$ can be considered as families of maps from $\Delta(C) \times \Delta(D)$ to $\m{R}$. They are both equi-Lipschitz, thus for these families, pointwise convergence and uniform convergence are equivalent. Gensbittel and Renault \cite{GR12} proved that $(v_n)$ converges, and asked whether $(v_{\lambda})$ converges. By Remark \ref{compactcase}, Assumptions 2 and 3 hold, and from Theorem \ref{th2} we deduce the following result:
\begin{proposition}
In a Markov chain game with incomplete information on both sides, for all $p_1 \in \Delta(C) \times \Delta(D)$, $(v_n(p_1))$ and $(v_{\lambda}(p_1))$ converge to the same limit.
\end{proposition}
\section{An example}
In this section, we prove the following theorem:
\begin{theorem}
There exists a hidden stochastic game such that for some initial state $k_1 \in K$ known by both players, $(v_{\lambda}(k_1))$ and $(v_n(k_1))$ converge to distinct limits.
\end{theorem}
\begin{remark}
As proved in Ziliotto \cite[Section 4, p. 21]{Z13}, this hidden stochastic game can be adapted, in order to get an example of a stochastic game with compact action sets and finite state space, such that for some initial state $k_1 \in K$, $(v_{\lambda}(k_1))$ and $(v_n(k_1))$ converge to distinct limits. It is also possible to build an example of a hidden stochastic game such that for some initial state $k_1 \in K$ known by both players, $(v_{\lambda}(k_1))$ converges but $(v_n(k_1))$ does not, and conversely, an example where $(v_n(k_1))$ converges but $(v_{\lambda}(k_1))$ does not.
\end{remark}
Before going to the proof, we provide some piece of intuition. In Ziliotto \cite{Z13}, a hidden stochastic game $\Gamma$ is constructed, in which neither $(v_{\lambda}(k_1))$ nor $(v_n(k_1))$ converges, where $k_1$ is an initial state known by both players. In the discounted game, there exists optimal stationary strategies (that is, strategies which only depend on the common belief about the current state). In this example, a stationary strategy for Player 1 (resp. 2) is equivalent to the choice of an integer $a \in r \m{N}$ (resp. $b \in 2r \m{N}$). Apart from the fact that Player 2's set of stationary strategies is smaller, the game is symmetric. In $\Gamma_{\lambda}(k_1)$, the optimal choice for both players is to choose $m$ as close as possible to $-\ln(\sqrt{2 \lambda})/\ln(2)$. For some  $\lambda$, the closest integer lies in $r(2\m{N}+1)$, and Player 1 has an advantage, whereas for some other discount factors, it lies in $2 r \m{N}$, and Player 1 has no advantage. This is why $(v_{\lambda}(k_1))$ oscillates (between $1/2$ and some $l \in (1/2,1]$). In $\Gamma_n(k_1)$, there may not exist optimal stationary strategies. Depending on the stage of the game $m \in \left\{1,...n \right\}$, the optimal integer for Player 1 lies in $2 r\m{N}$, or in $r(2\m{N}+1)$. Thus, according to the stage of the game, Player 1 may or may not have an advantage. This is in sharp contrast to the discounted game $\Gamma_{\lambda}(k_1)$, in which either Player 1 always has an advantage at any stage of the game, or it never has one. That is why we believe that in this example, $\liminf_{n \rightarrow+\infty} v_{n}(k_1)>1/2$ (but we were not able to prove it). Nonetheless, one can construct a hidden stochastic game $\Gamma^1$, very similar to $\Gamma$, in which $\liminf_{n \rightarrow+\infty} v^1_{n}(k_1)>1/2$ (this corresponds to Step \ref{G1} of the proof). In Step \ref{G2}, we construct a hidden stochastic game $\Gamma^2$, which only difference with $\Gamma^1$ is that Player 2's set of stationary strategies is equivalent to $r(2\m{N}+1)$, instead of $2r \m{N}$. In $\Gamma^2$, we also have $\liminf_{n \rightarrow+\infty} v^2_{n}(k_2)>1/2$, where $k_2$ is an initial state known by both players. In Step \ref{G3}, we define the hidden stochastic game $\Gamma^3$, in which starting from an initial state $\omega_3$, Player 2 chooses between playing $\Gamma^1(k_1)$ or $\Gamma^2(k_2)$. In $\Gamma^3_{\lambda}(\omega_3)$, Player 1 does not have an advantage, because the optimal integer is the same at any stage of the game (either it always lies in $2r \m{N}$, or it always lies in $r(2 \m{N}+1)$). Thus, $\lim_{\lambda \rightarrow 0} v^3_{\lambda}(\omega_3)=1/2$, but $\liminf_{n \rightarrow+\infty} v^3_{n}(\omega_3)>1/2$. It is then straightforward to construct in Step \ref{G4} a final example $\Gamma^4$ which proves the theorem.
\\
\begin{step} \label{G1} \ The game $\Gamma^1$.
\\

In Ziliotto \cite[Section 3, p.12-19]{Z13}, for some $r \geq 2$, a hidden stochastic game $\Gamma$ with $3r+4$ states, two signals $\left\{D,D'\right\}$ and two actions $\left\{C,Q\right\}$ for each player is constructed, such that for some initial state $k_1=1^{++} \in K$ known by both players, and for all $\lambda \in (0,1]$, the following properties hold:
\\
\\
\textbf{Properties of $\Gamma_{\lambda}(k_1)$}
\begin{enumerate}[(1)]
\item \label{eqgame}
The game $\Gamma_{\lambda}(k_1)$ (that is, the $\lambda$-discounted game starting from state $k_1$, and players know the initial state) has the same value as the one-shot game $G_{\lambda}$, with action set $r \m{N}$ for Player 1, $2 r \m{N}$ for Player 2, and payoff function
\begin{equation} \label{oneshot}
g_{\lambda}(a,b):=\frac{1-f_{\lambda}(b)}{1-f_{\lambda}(a)f_{\lambda}(b)},
\end{equation}
where 
\begin{equation*}
f_{\lambda}(n):=\frac{(1-2^{-n})(1-\lambda^2)}{1+2^{n+1}\lambda(1-\lambda)^{-n}-\lambda} \in [0,1).
\end{equation*}
\\
\item \label{P1win}
$v_{\lambda}(k_1) \geq 1/2$
\\
\item \label{vn1}
For $m \in \m{N}^*$, define $n_m:=2^{4rm+2r+1}$. Then
\begin{equation*}
\liminf_{m \rightarrow +\infty} v_{n_m}(k_1) \geq 3/4.
\end{equation*}
\item \label{draw}
Consider the one-shot game with action set $r \m{N}$ for Player 1 and 2, and payoff function $g_{\lambda}$. The value of this game converges to $1/2$ when $\lambda$ goes to 0.
\end{enumerate}
\vspace{0.25cm}
Property (\ref{eqgame}) corresponds to \cite[Proposition 3.3, p. 15]{Z13}. Property (\ref{vn1}) follows from the proof of \cite[Theorem 3.6, p. 18-19]{Z13}. In $G_{\lambda}$, a dominant strategy for Player 1 (resp. 2) is to maximize $f_{\lambda}$ over $r \m{N}$ (resp. $2 r \m{N}$). Consequently, if the action set of Player 2 is changed into $r \m{N}$, the value of this new game is equal to $\left[1+\max_{a \in r\m{N}} f_{\lambda}(a)\right]^{-1}$. This quantity is greater than $1/2$, thus Property (\ref{P1win}) holds. By \cite[Lemma 2.4, p. 10]{Z13}, this quantity goes to $1/2$ as $\lambda$ goes to 0, thus Property (\ref{draw}) holds.

Define another hidden stochastic game $\Gamma^1$, by adding a new state $\omega_1$ to $\Gamma$ (the action and signal sets are unchanged). The game is described by the following figure:
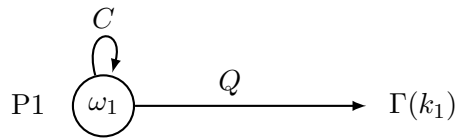
\begin{figure}[H]
\centering
   \caption{The game $\Gamma^1$}
   \vspace{0.3cm}
\begin{tikzpicture}[>=stealth',shorten >=1pt,auto,node distance=4cm,thick,main
 node/.style={circle,draw,font=\Large\bfseries}]

\node [draw,text width=0.4cm,text centered,circle] (O) at (0,0) {$\omega_1$};

\node [text width=0.4cm,text centered,circle] (A) at (4,0) {$\Gamma(k_1)$};

\draw[->,>=latex] (O) to[] node[midway, above left] {$Q$}(A);
\draw[->,>=latex] (O) edge[loop above] node[midway, above] {$C$}(O);

\draw(O)+(-1,-0.25) node[above]{P1};
\end{tikzpicture}
\end{figure}
The payoff in state $\omega_1$ is $1/2$, and Player 1 controls $\omega_1$  (that is, the transition in $\omega_1$ is independent of Player 2's actions). When Player 1 plays action $C$ in state $\omega_1$, the game remains in state $\omega_1$ with probability 1, and when Player 1 plays action $Q$, the game goes to state $k_1$ with probability 1, and the game $\Gamma(k_1)$ is played.

We have the following proposition:
\begin{proposition} \label{propG1} \mbox{}
\begin{enumerate}[(i)]
\item \label{propG1l}
The game $\Gamma^1_{\lambda}(\omega_1)$ has the same value as the one-shot game $G^1_{\lambda}$, with action set $r \m{N}$ for Player 1, $2 r \m{N}$ for Player 2, and payoff function $g^1_{\lambda}(a,b):=\lambda/2+(1-\lambda) g_{\lambda}(a,b)$.
\item \label{propG1n}
The sequence of values $(v^1_n)$ of $\Gamma^1_n$ satisfies
\begin{equation*}
\liminf_{n \rightarrow +\infty} v^1_n(\omega_1)>1/2.
\end{equation*}
\end{enumerate}
\end{proposition}
\begin{proof} \mbox{}
\begin{enumerate}[(i)]
\item
It follows from (\ref{P1win}) that an optimal strategy for Player 1 in $\Gamma^1_{\lambda}(\omega_1)$ is to play $Q$ at stage 1, then an optimal strategy in $\Gamma_{\lambda}(k_1)$, which proves the result.
\\
\item
Let $\epsilon>0$. By (\ref{vn1}), there exists $m_0 \in \m{N}^*$ such that for all $m \geq m_0$, $v_{n_m}(k_1) \geq 2/3$. Let $n \geq n_{m_0}$, and let $m \geq m_0$ such that
$n_m \leq n \leq n_{m+1}$. Define the following strategy for Player 1 in $\Gamma^1_{n}(\omega_1)$: play $C$ (thus stay in $\omega_1$) until stage $n-n_m$, play $Q$ at stage $n-n_m$, then
from stage $n-n_m+1$, play an optimal strategy in $\Gamma_{n_m}(k_1)$. This strategy guarantees
\begin{equation*}
\frac{1}{n}\left[\frac{1}{2}\left(n-n_m\right)+ \frac{2}{3} n_m \right] \geq \frac{1}{2}+\frac{1}{6} \frac{n_m}{n} \geq \frac{1}{2}+\frac{2^{-4r}}{6},
\end{equation*}
thus $\liminf_{n \rightarrow+\infty} v_n(\omega_1)>1/2$.
\end{enumerate}
\end{proof}
\end{step}
\vspace{0.5cm}
\begin{step} \label{G2} \ The game $\Gamma^2$.
\\

One can construct a hidden stochastic game $\Gamma$ similar to the example in \cite[Section 3, p.12-19]{Z13}, with $3r+4$ states, two signals $\left\{D,D'\right\}$ and two actions $\left\{C,Q\right\}$ for each player, such that for some initial state $k_1 \in K$ known by both players, the following properties hold:
\begin{enumerate}[(1)]
\item
The game $\Gamma_{\lambda}(k_1)$ has the same value as the one-shot game $G_{\lambda}$, with action set $r \m{N}$ for Player 1, $r(2\m{N}+1)$ for Player 2, and payoff function $g_{\lambda}$ defined by equation (\ref{oneshot}).
\\
\item 
$v_{\lambda}(k_1) \geq 1/2$
\\
\item
For $m \in \m{N}^*$, define $n_m:=2^{4rm+1}$. Then
\begin{equation*}
\liminf_{m \rightarrow +\infty} v_{n_m}(k_1) \geq 2/3.
\end{equation*}
\end{enumerate}
Using the same construction as in the previous step, adding one more state $\omega_2$ to $\Gamma$, we obtain a hidden stochastic game $\Gamma^2$ which satisfies the equivalent of Proposition \ref{propG1}:
\begin{proposition} \label{propG2} \mbox{}
\begin{enumerate}[(i)]
\item \label{propG2l}
The game $\Gamma^2_{\lambda}(\omega_2)$ has the same value as the one-shot game $G^2_{\lambda}$, with action set $r \m{N}$ for Player 1, $r (2\m{N}+1)$ for Player 2, and payoff function $g^2_{\lambda}(a,b):=\lambda/2+(1-\lambda) g_{\lambda}(a,b)$.
\item \label{propG2n}
The sequence of values $(v^2_n)$ of $\Gamma^2_n$ satisfies
\begin{equation*}
\liminf_{n \rightarrow +\infty} v^2_n(\omega_2)>1/2.
\end{equation*}
\end{enumerate}
\end{proposition}
\end{step}
\vspace{0.5cm}
\begin{step} \label{G3} \ The game $\Gamma^3$.
\\

Denote $K_1$ (resp. $K_2$) the state space of $\Gamma^1$ (resp. $\Gamma^2$), and define the hidden stochastic game with state space $K:=K_1 \cup K_2 \cup \left\{\omega_3 \right\}$, action sets $I:=J:=\left\{C,Q\right\}$, signal set $A:=\left\{D,D'\right\}$. The game is described by the following figure:
\begin{figure}[H]
\centering
   \caption{The game $\Gamma^3$}
   \vspace{0.3cm}
\begin{tikzpicture}[>=stealth',shorten >=1pt,auto,node distance=4cm,thick,main
 node/.style={circle,draw,font=\Large\bfseries}]

\node [draw,text width=0.4cm,text centered,circle] (O) at (0,0) {$\omega_3$};

\node [text width=0.4cm,text centered,circle] (A) at (4,1) {$\Gamma^1(\omega_1)$};
\node [text width=0.4cm,text centered,circle] (B) at (4,-1) {$\Gamma^2(\omega_2)$};

\draw[->,>=latex] (O) to[] node[midway, above left] {$C$}(A);
\draw[->,>=latex] (O) to[] node[midway, below left] {$Q$}(B);

\draw(O)+(-1,-0.25) node[above]{P2};
\end{tikzpicture}
\end{figure}
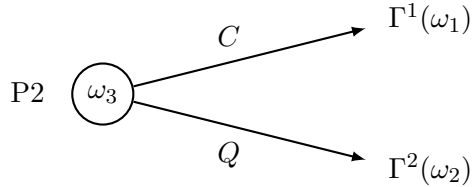
The payoff is $1/2$ in state $\omega_3$, and Player 2 controls this state. If Player 2 plays $C$ in $\omega_3$, in the next stage the game $\Gamma^1(\omega_1)$ is played, and if he plays $Q$, in the next stage the game $\Gamma^2(\omega_2)$ is played.

\begin{proposition} \mbox{}
\begin{enumerate}[(i)]
\item
The game $\Gamma^3_{\lambda}(\omega_3)$ has the same value as the one-shot game $\widetilde{G}_{\lambda}$, with action set $r \m{N}$ for Player 1 and Player 2, and payoff function $\widetilde{g}_{\lambda}:=\lambda(2-\lambda)/2+(1-\lambda)^2 g_{\lambda}$, where $g_{\lambda}$ is described by equation (\ref{oneshot}). In particular, $(v_{\lambda}(\omega_3))$ converges to $1/2$. 
\item
The sequence of values $(v^3_n)$ of $\Gamma^3_n$ satisfies
\begin{equation*}
\liminf_{n \rightarrow+\infty} v^3_n(\omega_3)>1/2.
\end{equation*}
\end{enumerate}
\end{proposition}
\begin{proof} \mbox{}
\begin{enumerate}[(i)] 
 \item
The first point follows from Proposition \ref{propG1} \ref{propG1l} and Proposition \ref{propG2} \ref{propG2l}, and the second point is a consequence of Property (\ref{draw}) in Step \ref{G1}.
\\
\item
This is a consequence of Proposition \ref{propG1} \ref{propG1n} and Proposition \ref{propG2} \ref{propG2n}.
\end{enumerate}
\end{proof}
\end{step}
\vspace{0.5cm}
\begin{step} \label{G4} \ Final example and proof of the theorem.
\\

Let $x$ such that $1/2<x<\liminf_{n \rightarrow+\infty} v^3_n(\omega_3)$. Define a hidden stochastic game $\Gamma^4$ by adding two more state $\omega_4$ and $x^*$ to $\Gamma^3$, as described in the following figure:
\begin{figure}[H]
\centering
   \caption{The game $\Gamma^4$}
   \vspace{0.3cm}
\begin{tikzpicture}[>=stealth',shorten >=1pt,auto,node distance=4cm,thick,main
 node/.style={circle,draw,font=\Large\bfseries}]

\node [draw,text width=0.4cm,text centered,circle] (O) at (0,0) {$\omega_4$};

\node [text width=0.4cm,text centered,circle] (A) at (4,1) {$\Gamma^3(\omega_3)$};
\node [draw,text width=0.4cm,text centered,circle] (B) at (4,-1) {$x^*$};

\draw[->,>=latex] (O) to[] node[midway, above left] {$C$}(A);
\draw[->,>=latex] (O) to[] node[midway, below left] {$Q$}(B);

\draw(O)+(-1,-0.25) node[above]{P2};
\end{tikzpicture}
\end{figure}
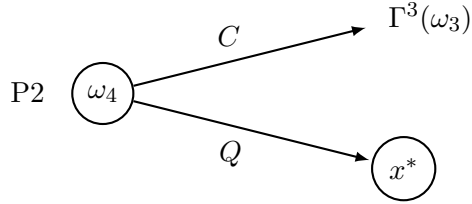
In state $\omega_4$, Player 2 has two options: play $C$ and play the game $\Gamma^3(\omega_3)$ from the next stage, or play $Q$ and get payoff $x$ forever.
Because $\lim_{\lambda \rightarrow 0} v^3_{\lambda}(\omega_3)=1/2<x$, for $\lambda$ small enough, playing action $C$ at stage 1 in $\Gamma^4_{\lambda}(\omega_4)$ is optimal for Player 2. Thus
\begin{equation*}
\lim_{\lambda \rightarrow 0} v^4_{\lambda}(\omega_4)=1/2. 
\end{equation*}
Because $\liminf_{n \rightarrow+\infty} v^3_n(\omega_3)>x$, for $n$ big enough, playing action $Q$ at stage 1 in $\Gamma^4_{n}(\omega_4)$ is optimal for Player 2. Thus
\begin{equation*}
\lim_{n \rightarrow+\infty} v^4_n(\omega_4)=x,
\end{equation*}
\end{step}
and the theorem is proved.

%


%
%
%

\section*{Acknowledgments.}
I am very grateful to my adviser J\'{e}r\^{o}me Renault for helpful discussions. This research was supported by the ANR Jeudy (ANR-10-BLAN 0112) and the GDR 2932.
\bibliography{biblio3}
\end{document}